\documentclass[a4paper,12pt]{article}
\usepackage[T2A]{fontenc}                       
\usepackage[cp1251]{inputenc}            
\usepackage[all]{xy}
\usepackage[english]{babel}
\usepackage{amssymb}
\usepackage{cite}
\usepackage{cmap}
\usepackage{latexsym}
\usepackage{amsmath}
\usepackage{enumerate}
\usepackage{amsmath, amsthm, amscd, amsfonts, amssymb, graphicx, color}
 \usepackage[left=2.5cm,right=1.5cm,top=1.3cm,bottom=2cm]{geometry}
\usepackage{indentfirst}

\newtheorem{thm}{Theorem}
\newtheorem{cor}{Corollary}[thm]
\newtheorem{lem}[thm]{Lemma}
\newtheorem{prop}[thm]{Proposition}

\newtheorem{defn}{Definition}

\newtheorem{example}{Example}[section]
\newtheorem{quest}{Question}[section]

\begin{document}

\begin{center}

\phantom{}

{\Large On the  Length of a Maximal Subgroup of a Finite Group\footnote{Supported by BRFFR $\Phi23\textrm{PH}\Phi\textrm{-}237$}}

\phantom{}

Viachaslau I. Murashka and Alexander F. Vasil'ev

\{mvimath@yandex.ru, formation56@mail.ru\}

\phantom{}

Faculty of Mathematics and Technologies of Programming,
Francisk Skorina Gomel State University, Sovetskaya 104, Gomel,
246028, Belarus

\end{center}

\begin{abstract}
  For a finite group $G$ and its maximal subgroup $M$ we proved that the generalized Fitting height of $M$ can't be less by 2 than the generalized Fitting height of $G$ and the non-$p$-soluble length of $M$ can't be less by 1 than the non-$p$-soluble length of $G$. We constructed  a hereditary saturated formation $\mathfrak{F}$ such that $\{n_\sigma(G, \mathfrak{F})-n_\sigma(M, \mathfrak{F})\mid G$ is finite $\sigma$-soluble and $M$ is a maximal subgroup of $G\}=\mathbb{N}\cup\{0\}$ where $n_\sigma(G, \mathfrak{F})$ denotes the $\sigma$-nilpotent length of the $\mathfrak{F}$-residual of  $G$.
This construction shows the results about  the generalized lengths of maximal subgroups published in Math. Nachr. (1994) and Mathematics (2020) are not correct.


\end{abstract}

\section*{Introduction and the Main results}

All groups considered here are finite. One way to study the structure of finite groups is to study their given normal series. An important parameter of such series is their length. For example the derived length, the nilpotent length and the $p$-length encode information about the structure of a group. Note that the series defining the nilpotent length were used for computations in soluble (polycyclic) groups \cite{Cannon2004}.
One of the main disadvantages of the above mentioned lengths is that they are not defined for all groups.
Khukhro and Shumyatsky \cite{Khukhro2015a, Khukhro2015b} introduced the following lengths associated with every group.

\begin{defn}[Khukhro, Shumyatsky]
  $(1)$ The generalized Fitting height $h^*(G)$ of a finite group $G$ is the least number $h$ such that $\mathrm{F}_h^* (G) = G$, where $\mathrm{F}_{(0)}^* (G) = 1$, and $\mathrm{F}_{(i+1)}^*(G)$ is the inverse image of the generalized Fitting subgroup $\mathrm{F}^*(G/\mathrm{F}^*_{(i)} (G))$.

  $(2)$  Let $p$ be a prime, $1=G_0\leq G_1\leq\dots\leq G_{2h+1}=G$ be the shortest normal series in which for $i$ odd the factor $G_{i+1}/G_i$ is $p$-soluble $($possibly trivial$)$,
and for $i$ even the factor $G_{i+1}/G_i$ is a $($non-empty$)$ direct product of nonabelian simple
groups. Then $h=\lambda_p(G)$ is called the non-$p$-soluble length of a group $G$.

$(3)$  $\lambda_2(G)=\lambda(G)$ is the nonsoluble length of a group $G$.
\end{defn}

Note that if $G$ is a soluble group, then $h^*(G)=h(G)$ is the nilpotent length of $G$.
For the properties and applications of these lengths see \cite{Fumagalli2019, Guralnick2020, KHUKHRO2015, Khukhro2015a, Khukhro2015b, Murashka2023}. From \cite[Theorem 5.6]{Fumagalli2019} it follows that for a group $G$ and its  subgroup $H$ the differences $h^*(G)-h^*(H)$ and $\lambda_p(G)-\lambda_p(H)$ are not bounded from below by a constant. Doerk \cite[Satz 1]{Doerk1994} proved that the difference of the nilpotent lengths of a soluble group and its  maximal subgroup can be only 0, 1 or 2. The analogues of this result were obtained for the $\pi$-length  of a $\pi$-soluble group \cite{Monakhov2009} and the $\sigma$-nilpotent length  of a $\sigma$-soluble group \cite{math8122165}. Here we prove

\begin{thm}\label{thm0}
  Let $M$ be a maximal subgroup of a group $G$ and $p$ be a prime.  Then
  $$h^*(G)-h^*(M)\leq 2, \lambda(G)-\lambda(M)\leq 1\textrm{ and }\lambda_p(G)-\lambda_p(M)\leq 1.$$
\end{thm}

This theorem is the consequence of two general results obtained via the functorial method. According to Plotkin \cite{Plotkin1973} a functorial   is a function $\gamma$ which assigns to each group $G$ its   subgroup $\gamma(G)$   satisfying $f(\gamma(G)) =\gamma(f(G)) $ for any isomorphism $f: G \to G^*$.
From \cite[p. 27 and Proposition 3.2.3]{Gardner2003} follows the following definition:
\begin{defn}
A functorial $\gamma$ is called a hereditary Plotkin radical if it satisfies:

 $(P1)$ $f(\gamma(G))\subseteq \gamma(f(G))$ for every epimorphism $f: G\to G^*$.

$(P2)$  $\gamma(G)\cap N=\gamma(N)$ for every $N\trianglelefteq G$.
\end{defn}

Note that the $\mathfrak{F}$-radical for a Fitting formation is a hereditary Plotkin radical.  Recall \cite{Plotkin1973} that  for functorials
  $\gamma_1$ and $\gamma_2$  the upper product $\gamma_2\star\gamma_1$ is  defined by $(\gamma_2\star\gamma_1)(G)/\gamma_2(G)=\gamma_1(G/\gamma_2(G))$. This operation is an associative one.
With every functorial one can associate the following length.

   \begin{defn}[\hspace{-1pt}{\cite[Definition 2.4]{Murashka2023}}]\label{hun}
  Let $\gamma$ be a functorial. Then the $\gamma$-series of $G$ is defined  starting from $\gamma_{(0)}(G)=1$, and then by induction $\gamma_{(i+1)}(G)=(\gamma_{(i)}\star\gamma)(G)$. The least number $h$ such that $\gamma_{(h)}(G)=G$ is defined to be the $\gamma$-length $h_\gamma(G)$ of $G$. If there is no such number, then $h_\gamma(G)=\infty$.
\end{defn}

If $\gamma=\mathrm{F}$ assigns to every group its Fitting subgroup, then the $\gamma$-length is just the Fitting height (length) and for a group $G$ is denoted by $l(G)$ or $h(G)$. For $\gamma=\mathrm{F}^*$ we get the generalized Fitting height. One of our main results is

\begin{thm}\label{thm1}
  Let $\gamma$   be a hereditary Plotkin radical which satisfies   $\mathrm{F}^*(G)\subseteq\gamma(G)$ for any group $G$ with $h_\gamma(G)<\infty$. If $M$ is a maximal subgroup of a group $G$ and  $h_\gamma(G), h_\gamma(M)<\infty$, then  $h_\gamma(G)-h_\gamma(M)\leq 2$.
\end{thm}

From \cite[Theorem 3.1 and Corollary 3.4(A)]{Baer1966} it follows that if $\gamma$ is a hereditary Plotkin radical iff $\mathfrak{F}=(G\mid \gamma(G)=G)$ is a $Q$-closed Fitting class and $\gamma$ is the $\mathfrak{F}$-radical. The example of a $Q$-closed Fitting class of soluble groups which is not a formation follows from    \cite[IX, Examples 2.21(b)]{Doerk1992}. Theorem \ref{thm1} gives the analogues of Doerk's result  for any  $Q$-closed Fitting class of soluble groups.

\begin{cor}\label{cor}
   Let $\mathfrak{F}$ be a $Q$-closed Fitting class of soluble groups, $\gamma$ assigns to every group its $\mathfrak{F}$-radical and $\pi=\pi(\mathfrak{F})$. If $M$ is a maximal subgroup of a soluble $\pi$-group $G$, then
  $h_\gamma(G)-h_\gamma(M)\leq 2$.
\end{cor}

Note that $h(H)\leq h(G)$ for any soluble group $G$ and its subgroup $H$. Hence from Theorem~\ref{thm1} for $\gamma=\mathrm{F}$ follows

\begin{cor}[\hspace{-1pt}\cite{Doerk1994}]\label{cor1}
   Let $M$ be a maximal subgroup of a soluble group $G$. Then
  $h(G)-h(M)\in\{0, 1, 2\}$.
\end{cor}

 Let $\sigma=\{\sigma_i\mid i\in I\}$ be a partition of  the set of all primes $\mathbb{P}$. Recall \cite{skiba2015sigma} that a group $G$ is called $\sigma$-\emph{soluble} if for every its chief factor $H/K$ there exists $\sigma_i\in\sigma$ such that $H/K$ is a $\sigma_i$-group (i.e. all prime divisors of $|H/K|$ belong to $\sigma_i$); $\sigma$-\emph{nilpotent}  if it has a normal Hall $\sigma_i$-subgroup for every $\sigma_i\in\sigma$. The greatest normal $\sigma$-nilpotent subgroup of $G$ is denoted by $\mathrm{F}_\sigma(G)$. The $\gamma$-length of $G$ for $\gamma=\mathrm{F}_\sigma$ is denoted by $l_\sigma(G)$. Note that a group is $\sigma$-soluble iff $l_\sigma(G)<\infty$.

\begin{cor}[\hspace{-1pt}\cite{math8122165}]\label{cor2}
   Let $\sigma$ be a partition of $\mathbb{P}$ and $M$ be a maximal subgroup of a $\sigma$-soluble group $G$. Then
  $l_\sigma(G)-l_\sigma(M)\in\{0, 1, 2\}$.
\end{cor}

According to \cite[p. 27 and Proposition 3.2.3]{Gardner2003} a hereditary Kurosh-Amitsur radical can be defined in the following way:

\begin{defn}
  A hereditary Plotkin radical $\gamma$ is called a hereditary Kurosh-Amitsur radical if it satisfies  $(P3)$: $\gamma(G/\gamma(G))\simeq1$ for every group $G$.
\end{defn}

For a class of simple groups $\mathfrak{J}$ the greatest normal subgroup $\mathrm{O}_\mathfrak{J}(G)$ of $G$ all whose composition factors belong to $\mathfrak{J}$ is the example of hereditary Kurosh-Amitsur radical. Kurosh-Amitsur radicals (of groups) were studied in \cite{Krempa2011}.

\begin{thm}\label{thm2}
  Let $\rho$   be a hereditary Kurosh-Amitsur radical which contains the soluble radical in every group and $\gamma=\rho\star\mathrm{F}^*\star\rho$.  If $M$ is a maximal subgroup of a group $G$ and $h_\gamma(G), h_\gamma(M)<\infty$, then  $h_\gamma(G)-h_\gamma(M)\leq 1$.
\end{thm}

Recall that for a formation $\mathfrak{F}$ and a group $G$ the $\mathfrak{F}$-residual of $G$ is denoted by $G^\mathfrak{F}$. The nilpotent and $\sigma$-nilpotent lengths of the $\mathfrak{F}$-residual are denoted by $n_\mathfrak{F}(G)$ \cite{BallesterBolinches1994} and  $n_\sigma(G, \mathfrak{F})$ \cite{math8122165} respectively.
Let $\mathfrak{F}$ be a hereditary saturated formation. In  \cite{BallesterBolinches1994} it was claimed that
$n_\mathfrak{F}(G)-n_\mathfrak{F}(M)\in\{0, 1, 2\}$ for any  soluble group $G$ and its maximal subgroup $M$. For a partition $\sigma$ of $\mathbb{P}$ in the paper \cite{math8122165} it was proved that $n_\sigma(G, \mathfrak{F})-n_\sigma(M, \mathfrak{F})\in\{0, 1, 2\}$
for any  $\sigma$-soluble group $G$ and its maximal subgroup $M$. Our next result shows that the two above mentioned facts are \textbf{\emph{wrong}}.

\begin{thm}\label{thm3}
  Let $\sigma$ be a partition of $\mathbb{P}$ with $|\sigma|>1$. Then there exists a hereditary saturated formation $\mathfrak{F}=\mathfrak{F}(\sigma)$ of soluble groups such that
  \begin{center}
  $\{n_\sigma(G, \mathfrak{F})-n_\sigma(M, \mathfrak{F})\mid G$ is $\sigma$-soluble and $M$ is a maximal subgroup of $G\}=\mathbb{N}\cup\{0\}.$
  \end{center}
In particular, there exists  a hereditary saturated formation $\mathfrak{F}$ such that
\begin{center}
$\{n_\mathfrak{F}(G)-n_\mathfrak{F}(M)\mid G$ is soluble and $M$ is a maximal subgroup of $G\}=\mathbb{N}\cup\{0\}.$
\end{center}
\end{thm}

\section{Preliminaries}

All unexplained notations and terminologies are standard. The reader is referred to \cite{Doerk1992} if necessary. Recall that $\mathbb{N} $ and $ \mathbb{P}$ denote the sets of all natural and prime numbers respectively.

 Recall that a \emph{class of groups} is a collection $\mathfrak{F}$ of groups with the property that
if $G\in\mathfrak{F} $ and if $H \simeq  G$, then $H \in \mathfrak{F}$; a \emph{formation} is a class of groups $\mathfrak{F}$ which is   closed  under taking epimorphic images (i.e. from $G\in\mathfrak{F}$ and $N\trianglelefteq G$ it follows that $G/N\in\mathfrak{F}$)  and subdirect products (i.e. from $G/N_1\in\mathfrak{F}$ and $G/N_2\in\mathfrak{F}$ it follows that $G/(N_1\cap N_2)\in\mathfrak{F}$). A formation $\mathfrak{F}$ is called: hereditary if $H\in\mathfrak{F}$ whenever $H\leq G\in\mathfrak{F}$; saturated if $G\in\mathfrak{F}$ whenever $G/\Phi(G)\in\mathfrak{F}$. The smallest normal subgroup of $G$ with quotient in $\mathfrak{F}$ is called the $\mathfrak{F}$-residual of $G$. A group $G$ is called $p$-closed if it has a normal Sylow $p$-subgroup. The class of all soluble $p$-closed groups is the example of a hereditary saturated formation.

From \cite[Proposition 2.3 and Lemma 2.6]{Murashka2023}  the next result follows.

\begin{lem}\label{lem1}
  If $\gamma$ is a hereditary Plotkin radical, then $\gamma_{(n)}$ is a hereditary Plotkin radical for any $n\in\mathbb{N}$ and
  $\max\{h_\gamma(N), h_\gamma(G/N)\}\leq h_\gamma(G)\leq h_\gamma(G/N)+h_\gamma(N)$ for any $N\trianglelefteq G$.
\end{lem}

One of the characteristic properties of Kurosh-Amitsur radicals is the following

\begin{lem}\label{lem2}
  Let $\gamma$ be a hereditary Kurosh-Amitsur radical. Then $\gamma(G/N)=\gamma(G)/N$ for any $N\trianglelefteq G$ with $N\subseteq\gamma(G)$.
\end{lem}

\begin{proof}
  Assume that $\gamma(G)/N<\gamma(G/N)=H/N$. Then $1\not\simeq H/\gamma(G)=\gamma(H/\gamma(G))\subseteq\gamma(G/\gamma(G))\simeq1$, a contradiction.
\end{proof}

\section{Proves of the Main Results}

\subsection{Proof of Theorem \ref{thm1}}


Assume the contrary. Let a group $G$ be a minimal order counterexample. Hence $G$ has  a maximal subgroup $M$ with $h_\gamma(G)-h_\gamma(M)>2$. It is clear that $h_\gamma(G)\geq 3$.
Let $M_i=\gamma_{(i)}(M)$ and $G_i=\gamma_{(i)}(G)$. If $MG_1=G$, then $h_\gamma(G)-1=h_\gamma(G/G_{1})=h_\gamma(MG_1/G_{1})=h_\gamma(M/(M\cap G_{1}))\leq h_\gamma(M)$ by Lemma \ref{lem1}. It means that $h_\gamma(G)-h_\gamma(M)\leq 1$, a contradiction. Therefore $G_1\subseteq M$.

If $M_0=M$, then $G$ is a cyclic group of prime order and $h_\gamma(G)-h_\gamma(M)=1$, a contradiction. So $M_0\neq M$.
 Suppose that $ M_i\subseteq G_{i+1}\subseteq M$  and $M_i\neq M$ for some $ i\geq0$. At least it is true for $i=0$. Let prove that   $ M_{i+1}\subseteq G_{i+2}\subseteq M$ and $M_{i+1}\neq M$.

Note that $h_\gamma(G)>i+1$ and $h_\gamma(M)>i$. From $  M_i\subseteq G_{i+1}\subseteq G_{i+2}$ it follows that $M_i\subseteq M\cap G_{i+2}$.
If $G_{i+2}\not\leq M$, then by Definition \ref{hun} and Lemma \ref{lem1}
 $$h_\gamma(G)-(i+2)=h_\gamma(G/G_{i+2}) =h_\gamma(MG_{i+2}/G_{i+2})=h_\gamma(M/(M\cap G_{i+2}))\leq h_\gamma(M)-i.$$
 Therefore $h_\gamma(G)-h_\gamma(M)\leq (i+2)-i=2$, a contradiction.  Thus $G_{i+2}\subseteq M$.

 Now $G_{i+2}, M_{i+1}\trianglelefteq M$. Let $I=G_{i+2}\cap M_{i+1}\trianglelefteq M$. From $I\trianglelefteq M_{i+1}$ it follows that $\gamma_{(i+1)}(I)=I$ by Lemma \ref{lem1}.  From the other hand  $I\trianglelefteq G_{i+2}$ and $G_{i+1}=\gamma_{(i+1)}(G)=\gamma_{(i+1)}(G)\cap G_{i+2}=\gamma_{(i+1)}(G_{i+2})$ by $(P2)$ and Lemma \ref{lem1}. Thus  $I\leq G_{i+1}$ by $(P2)$.

 Let $F/G_{i+1}=\mathrm{F}^*(G/G_{i+1})$.  From $h_\gamma(G)<\infty$ it follows that $h_\gamma(G/G_{i+1})<\infty$. Therefore $F/G_{i+1}\subseteq \gamma(G/G_{i+1})=G_{i+2}/G_{i+1}$. Hence $F\leq G_{i+2}$.
 Now $$(M_{i+1}G_{i+1}/G_{i+1})\cap G_{i+2}/G_{i+1}=(M_{i+1} \cap G_{i+2})G_{i+1}/G_{i+1}=G_{i+1}/G_{i+1}\simeq 1.$$
 From \cite[X, Theorem 13.12]{Huppert1982} it follows that
 $$M_{i+1}G_{i+1}/G_{i+1}\subseteq C_{G/G_{i+1}}(G_{i+2}/G_{i+1})\subseteq C_{G/G_{i+1}}(F/G_{i+1})\subseteq F/G_{i+1}\subseteq G_{i+2}/G_{i+1}.$$
 Thus $M_{i+1}\subseteq  G_{i+2}$.
 If $M_{i+1}=M$, then $G_{i+2}=M<G$. By our assumption $M_i\neq M$. Hence $h_\gamma(M)=i+1$ and $h_\gamma(G)= i+3$. Therefore $h_\gamma(G)-h_\gamma(M)= 2$, a contradiction. Thus $M_{i+1}\neq M$.

 It means $M_{i}\subseteq  G_{i+1}\subseteq M$ and $M_i\neq M$ for every natural $i$.  Thus $h_\gamma(G)=\infty$, the contradiction.

\subsection{Proof of Corollary \ref{cor}}

Since $\mathfrak{F}$ is a $Q$-closed Fitting class of soluble groups, we see that $\gamma$ is a hereditary Plotkin radical by \cite[Theorem 3.1 and Corollary 3.4(A)]{Baer1966} and $\mathfrak{F}$ contains a group of order $p$ for any $p\in\pi$. It means that $h_\gamma(G)<\infty$ iff $G$  is a  soluble $\pi$-group and $\mathfrak{F}$ contains all nilpotent $\pi$-groups by \cite[IX, Theorem 1.9]{Doerk1992}. So $\mathrm{F}^*(G)=\mathrm{F}(G)\subseteq\gamma(G)$ for any group $G$ with $h_\gamma(G)<\infty$. Thus Corollary \ref{cor} directly follows from Theorem \ref{thm1}.

\subsection{Proof of Theorem \ref{thm2}}

Note that $\gamma$ is a hereditary Plotkin radical by \cite[Proposition 2.3]{Murashka2023}.

Assume the contrary. Let a group $G$ be a minimal order counterexample. Hence $G$ has  a maximal subgroup $M$ with $h_\gamma(G)-h_\gamma(M)>1$. It is clear that $h_\gamma(G)>1$.

If $M\gamma(G)=G$, then by Definition \ref{hun} and Lemma \ref{lem1}
$$h_\gamma(G)-1=h_\gamma(G/\gamma(G))=h_\gamma(M/(M\cap\gamma(G)))\leq h_\gamma(M).$$
Therefore $h_\gamma(G)-h_\gamma(M)\leq 1$, a contradiction. Hence $\gamma(G)\subseteq M$. Since $\rho$ satisfies $(P2)$, we see that $\rho(G)\subseteq\rho(M)$. Note that
$(M/\rho(G))/(\rho(M/\rho(G)))=(M/\rho(G))/(\rho(M)/\rho(G))\simeq M/\rho(M)$ and $\rho(G/\rho(G))\simeq 1$ by Lemma \ref{lem2} and $(P3)$.
From the definition of $\gamma$ it follows that  $\gamma(G)/\rho(G)=\gamma(G/\rho(G))$ and $\gamma(M/\rho(G))=\gamma(M)/\rho(G)$.
If $\rho(G)=M$, then $h_\gamma(G)-h_\gamma(M)=1-1=0$, a contradiction. Hence $h_\gamma(G/\rho(G))=h_\gamma(G)$ and
$h_\gamma(M/\rho(G))=h_\gamma(M)$. From our assumption it follows that $\rho(G)=1$.  So $\rho(\gamma(G))=1$. From $\gamma(G), \rho(M)\trianglelefteq M$ it follows that $\rho(M)\cap\gamma(G)\trianglelefteq \gamma(G)$. Hence $\rho(M)\cap\gamma(G)=\rho(\gamma(G))=1$. Now from \cite[X, Theorem 13.12]{Huppert1982} it follows that $$\rho(M)\subseteq C_G(\gamma(G))\subseteq C_G(\mathrm{F}^*(G))\subseteq\mathrm{F}^*(G)\subseteq\gamma(G).$$
It means that $\rho(M)=1$. Now $M_\mathfrak{S}=1$. Therefore   $\mathrm{F}^*(M)$ is the direct products of minimal normal non-abelian subgroups of $M$ by \cite[X, Definition 13.14 and Lemma 13.16]{Huppert1982}. Let $M_1$ be one of them. If $M_1\not\subseteq\mathrm{F}^*(G)$, then $M_1\cap \mathrm{F}^*(G)=1$. So $M_1\subseteq C_G(\mathrm{F}^*(G))\subseteq \mathrm{F}^*(G)$, a contradiction. Hence $\mathrm{F}^*(M)\subseteq\mathrm{F}^*(G)\subseteq \gamma(G)\subseteq M$. Thus $\mathrm{F}^*(M)=\mathrm{F}^*(G)$.

Since $\rho$ is a Kurosh-Amitsur radical and   $h_\gamma(G)>1$, we see that  $h_\gamma(G/\mathrm{F}^*(G))=h_\gamma(G)-1$. If $h_\gamma(M)>1$, then
$h_\gamma(M/\mathrm{F}^*(G))=h_\gamma(M/\mathrm{F}^*(M))=h_\gamma(M)-1$ and we get the contradiction with the initial assumption. Thus $h_\gamma(M)=1$.  It means that $M/\mathrm{F}^*(G)=\rho(M/\mathrm{F}^*(G))$. Therefore $\gamma_2(G)\not\subseteq M$.  Now $G/\gamma_{(2)}(G)=M/(M\cap\gamma_{(2)}(G))$.
So $1\simeq \rho(G/\gamma_{(2)}(G))=\rho(M/(M\cap\gamma_{(2)}(G)))$ by Lemma \ref{lem2} and definition of $\gamma$. From  $\gamma(G)\subseteq\gamma_{(2)}(G)\cap M$ and $M/\gamma(G)=\rho(M/\gamma(G))$ it follows that $M/(M\cap\gamma_{(2)}(G))=\rho(M/(M\cap\gamma_{(2)}(G)))\simeq 1$. Thus $h_\gamma(G)=2$ and $h_\gamma(M)=1$, the final contradiction.

\subsection{Proof of Theorem \ref{thm0}}

If $\gamma=\mathrm{F}^*$ then from Theorem \ref{thm1} it follows that $h^*(G)-h^*(M)\leq 2$ for any group $G$ and its maximal subgroup $M$.

Assume that $\rho$ is the $p$-soluble radical and $\gamma=\rho\star\mathrm{F}^*\star\rho$. Then $\gamma$ satisfies the assumptions of Theorem \ref{thm2}. Hence if $H$ is not a $p$-soluble group, then $h_\gamma(H)=\lambda_p(H)$ by \cite[Lemma 2.7]{Murashka2023}. Let a group $G$ be a minimal order group with a maximal subgroup $M$ such that $\lambda_p(G)-\lambda_p(M)>1$. It means that $G$ is a non-$p$-soluble group, $\lambda_p(G)>1$ and $M$ is $p$-soluble.  If $M\gamma(G)=G$, then from $\lambda_p(G/\gamma(G))\geq 1$ and $G/\gamma(G)\simeq M/(M\cap \gamma(G))$ it follows that a $p$-soluble group $M$ has a non-$p$-soluble composition factor, a contradiction. Thus $\gamma(G)\trianglelefteq M$. Hence  a $p$-soluble group $M$ has a non-$p$-soluble composition factor, the final contradiction. It means that  $\lambda_p(G)-\lambda_p(M)\leq 1$  and $\lambda(G)-\lambda(M)\leq 1$   ($\lambda=\lambda_2$) for any group $G$ and its maximal subgroup $M$.

\subsection{Proof of Theorem \ref{thm3}}
From $|\sigma|>1$ it follows that there exists $p\in\mathbb{P}$ such that $|\sigma\cap(\mathbb{P}\setminus\{p\})|>1$. Let $\mathfrak{F}$ be a formation of all $p$-closed soluble groups. Then $\mathfrak{F}$ is a hereditary saturated formation. Note that $n_\sigma(G, \mathfrak{F})-n_\sigma(M, \mathfrak{F})=0$ for every soluble $p$-closed group $G$ and its maximal subgroup $M$.

 For every $n>0$ there exists a sequence of not necessary different  primes $p=p_0, p_1, p_2,\dots, p_n$ such that every two of its consecutive elements belong to different elements of $\sigma$ and $p_i\neq p$ for all $i>0$. Let $G_1$ be a cyclic group of order $p$. Define a sequence of subgroups $G_i$ inductively.
Note that for $G_i$ there exists  a faithful irreducible module $V_i$ over $\mathbb{F}_{p_i}$ \cite[B, Theorem 10.3]{Doerk1992}. Let $G_{i+1}$ be the semidirect product of $V_i$ with $G_i$ corresponding to the action
of $G_i$ on $V_i$ as an $\mathbb{F}_pG_i$-module. Since $p_i$ and $p_{i-1}$ belong to different elements of $\sigma$ and $V_i$ is the unique minimal normal subgroup of $G_{i+1}$, we see that $\mathrm{F}_\sigma(G_{i+1})=V_i$.

Let $G=G_{n+1}$ and $M_i=V_iV_{i-1}\dots V_1$. Then $M_n$ is a maximal subgroup of $G$ and a $p'$-group. Hence $l_\sigma(M_n^\mathfrak{F})=l_\sigma(1)=0$. Note that $G$ has the unique chief series and $G_2\simeq G/(V_2V_{3}\dots V_n)$ is not $p$-closed. It means that $G^\mathfrak{F}=M_n$. Note that $M_i\trianglelefteq G_{i+1}$. Now $\mathrm{F}_\sigma(M_i)=\mathrm{F}_\sigma(G_{i+1})\cap M_i=V_i\cap M_i=V_i$. It means that $l_\sigma(M_n)=n$.
Therefore $n_\sigma(G, \mathfrak{F})-n_\sigma(M, \mathfrak{F})=n$. Since every soluble group is $\sigma$-soluble, we see that
  \begin{center}
  $\{n_\sigma(G, \mathfrak{F})-n_\sigma(M, \mathfrak{F})\mid G$ is $\sigma$-soluble and $M$ is a maximal subgroup of $G\}=\mathbb{N}\cup\{0\}.$
  \end{center}
In particular if $|\sigma_i|=1$ for every $\sigma_i\in\sigma$, then
\begin{center}
  $\{n_\mathfrak{F}(G)-n_\mathfrak{F}(M)\mid G$ is soluble and $M$ is a maximal subgroup of $G\}=\mathbb{N}\cup\{0\}.$
\end{center}

\section{Final Remarks and Open Questions}

Note that $h(H^\mathfrak{N})=h(H)-1$ for any non-unit soluble group  $H$ and $h(H^\mathfrak{N})=h(H)$ for a unit group $H$. If a unit group  is a maximal subgroup $M$ of $G$, then $G$ is cyclic and $n_\mathfrak{F}(G)-n_\mathfrak{F}(M)=0$. If $M$ is a non-unit  subgroup  of a soluble group $G$, then $h(M^\mathfrak{N})=h(M)-1$ and $h(G^\mathfrak{N})=h(G)-1$. Hence $n_\mathfrak{N}(G)-n_\mathfrak{N}(M)\in\{0, 1, 2\}$ for any soluble group $G$ and its maximal subgroup $M$ by Corollary \ref{cor1}.  That is why the main result of \cite{BallesterBolinches1994} is wrong not for all hereditary saturated formations. Therefore the following question seems natural:

\begin{quest}\label{ques1}
  Describe all hereditary saturated formations $\mathfrak{F}$ such that $n_\mathfrak{F}(G)-n_\mathfrak{F}(M)\in\{0, 1, 2\}$ for any soluble group $G$ and its maximal subgroup $M$.
\end{quest}

\begin{prop}\label{prop1}
  Let  $\mathfrak{F}$ be a hereditary saturated formation containing all nilpotent groups. Assume that there exists a constant $n$ such that $h(G)\leq n$ for any soluble $\mathfrak{F}$-group $G$. Then $n_\mathfrak{F}(G)-n_\mathfrak{F}(M)\leq n+1$ for any soluble group $G$ and its maximal subgroup $M$.
\end{prop}

\begin{proof}
  Note that $H^{\mathfrak{N}^n}\subseteq H^\mathfrak{F}$ for any group $H$. It means that $h(H)-h(H^\mathfrak{F})\leq n$ for any group $H$ by Lemma \ref{lem1}. If $h(G)=h(G^\mathfrak{F})$, then $G\simeq 1$ and has no maximal subgroups. Assume that $G\not\simeq 1$. Then $1\leq h(G)-h(G^\mathfrak{F})\leq n$, $h(M)-h(M^\mathfrak{F})\leq n$ and $h(G)-h(M)\leq 2$. So
  $(h(G)-h(G^\mathfrak{F}))-(h(M)-h(M^\mathfrak{F}))\geq 1-n$ or $n+1\geq h(G)-h(M)+n-1\geq h(G^\mathfrak{F})-h(M^\mathfrak{F})$. Thus $n_\mathfrak{F}(G)-n_\mathfrak{F}(M)\leq n+1$.
\end{proof}

\begin{example}
  There exists formations  $\mathfrak{F}$ for which the value $n_\mathfrak{F}(G)-n_\mathfrak{F}(M)$ is bouded but not by 2.
Let $p$ be a prime and $\mathfrak{F}$ be a class of all $p$-closed soluble groups of nilpotent length at most 3. Then $\mathfrak{F}$ is a hereditary saturated formation and $n_\mathfrak{F}(G)-n_\mathfrak{F}(M)\leq 4$ by Proposition \ref{prop1} for any soluble group $G$ and its maximal subgroup $M$.

Let $G_4$ and $M_3$ be the same as in the proof of Theorem \ref{thm3}. Note that $h(M)\leq 3$ and hence $M^\mathfrak{F}=1$. Therefore $n_\mathfrak{F}(G)-n_\mathfrak{F}(M)=3>2$. \end{example}

In the view of this example  the following question seems interesting:

\begin{quest}\label{ques2}
  Describe all hereditary saturated formations $\mathfrak{F}$ such that there exists a constant $n$ with $n_\mathfrak{F}(G)-n_\mathfrak{F}(M)\leq n$ for any soluble group $G$ and its maximal subgroup $M$. For such formation $\mathfrak{F}$ do there exists a constant $m$ with $h(G)\leq m$ for every soluble $\mathfrak{F}$-group $G$.
\end{quest}

Recall that $\mathfrak{N}_\sigma$ denotes the formation of all $\sigma$-nilpotent groups. With the help of  Corollary \ref{cor2} one can prove that  $n_\sigma(G, \mathfrak{N}_\sigma)-n_\sigma(M, \mathfrak{N}_\sigma)\in\{0, 1, 2\}$ for any $\sigma$-soluble group $G$ and its maximal subgroup $M$.

\begin{quest}
Consider analogues of Questions \ref{ques1} and \ref{ques2} for the  $n_\sigma(G, \mathfrak{F})$.
\end{quest}
\bibliographystyle{plain}
\bibliography{len2}

\begin{thebibliography}{10}

\bibitem{Baer1966}
R.~Baer.
\newblock Group theoretical properties and functions.
\newblock {\em Colloq. Math.}, 14:285--327, 1966.

\bibitem{BallesterBolinches1994}
A.~Ballester-Bolinches and M.~D. P\'erez-Ramos.
\newblock {A Note on the $\mathfrak{F}$-length of Maximal Subgroups in Finite
  Soluble Groups}.
\newblock {\em Math. Nachr.}, 166(1):67--70, 1994.

\bibitem{Cannon2004}
J.~J. Cannon, B.~Eick, and C.~R. Leedham-Green.
\newblock Special polycyclic generating sequences for finite soluble groups.
\newblock {\em J. Symb. Comput.}, 38(5):1445--1460, 2004.

\bibitem{Doerk1994}
K.~Doerk.
\newblock {\"Uber die nilpotente L\"ange maximaler Untergruppen bei endlichen
  aufl\"osbaren Gruppen}.
\newblock {\em Rend. Sem. Mat. Univ. Padova}, 91:20--21, 1994.

\bibitem{Doerk1992}
K.~Doerk and T.~O. Hawkes.
\newblock {\em {Finite Soluble Groups}}, volume~4 of {\em De Gruyter Exp.
  Math.}
\newblock De Gruyter, Berlin, New York, 1992.

\bibitem{Fumagalli2019}
F.~Fumagalli, F.~Leinen, and O.~Puglisi.
\newblock A reduction theorem for nonsolvable finite groups.
\newblock {\em Isr. J. Math.}, 232(1):231--260, 2019.

\bibitem{Gardner2003}
B.~J. Gardner and R.~Wiegandt.
\newblock {\em Radical theory of rings}.
\newblock Marcel Dekker New York, 2003.

\bibitem{Guralnick2020}
R.~M. Guralnick and G.~Tracey.
\newblock {On the generalized Fitting height and insoluble length of finite
  groups}.
\newblock {\em Bull. London Math. Soc.}, 52(5):924--931, 2020.

\bibitem{math8122165}
A.~Heliel, M.~Al-Shomrani, and A.~Ballester-Bolinches.
\newblock {On the $\sigma$-Length of Maximal Subgroups of Finite
  $\sigma$-Soluble Groups}.
\newblock {\em Mathematics}, 8(12), 2020.

\bibitem{Huppert1982}
B.~Huppert and N.~Blackburn.
\newblock {\em {Finite groups III}}, volume 243 of {\em Grundlehren Math.
  Wiss.}
\newblock Springer-Verlag, Berlin, Heidelberg, 1982.

\bibitem{Khukhro2015a}
E.~I. Khukhro and P.~Shumyatsky.
\newblock Nonsoluble and non-p-soluble length of finite groups.
\newblock {\em {Isr. J. Math.}}, 207(2):507--525, 2015.

\bibitem{Khukhro2015b}
E.~I. Khukhro and P.~Shumyatsky.
\newblock On the length of finite factorized groups.
\newblock {\em {Ann. Mat. Pura Appl.}}, 194(6):1775--1780, 2015.

\bibitem{KHUKHRO2015}
E.~I. Khukhro and P.~Shumyatsky.
\newblock On the length of finite groups and of fixed points.
\newblock {\em {Proc. Amer. Math. Soc.}}, 143(9):3781--3790, 2015.

\bibitem{Krempa2011}
J.~Krempa and I.~A. Malinowska.
\newblock {On Kurosh-Amitsur radicals of finite groups}.
\newblock {\em An. Stiint. Univ. “Ovidius” Constanta Ser. Mat.},
  19(1):175--190, 2011.

\bibitem{Monakhov2009}
V.~S. Monakhov and O.~A. Shpyrko.
\newblock The nilpotent $\pi$-length of maximum subgroups in finite
  $\pi$-soluble groups.
\newblock {\em Vestnik Moskov. Univ. Ser. 1. Mat. Mekh.}, (6):3--8, 2009.

\bibitem{Murashka2023}
V.~I. Murashka and A.~F. Vasil'ev.
\newblock On the lengths of mutually permutable products of finite groups.
\newblock {\em Acta Math. Hungar.}, 170(1):412--429, 2023.

\bibitem{Plotkin1973}
B.~I. Plotkin.
\newblock Radicals in groups, operations on group classes and radical classes.
\newblock In {\em Selected questions of algebra and logic}, pages 205--244.
  Nauka, Novosibirsk, 1973.
\newblock In Russian.

\bibitem{skiba2015sigma}
A.~N. Skiba.
\newblock On $\sigma$-subnormal and $\sigma$-permutable subgroups of finite
  groups.
\newblock {\em J. Algebra}, 436:1--16, 2015.

\end{thebibliography}

\end{document}